\documentclass[a4paper]{article}

\usepackage[utf8]{inputenc}
\usepackage{subfigure}

\usepackage{lmodern}
\usepackage{geometry}
\usepackage{array}
\usepackage{graphicx}
\usepackage{amssymb, amsfonts, amsthm, amsmath}
\usepackage{enumerate}
\usepackage{cite}

\usepackage{soul}

\usepackage[english]{babel}
\usepackage[colorlinks]{hyperref}
\usepackage{tikz}
\usetikzlibrary{decorations}

\numberwithin{equation}{section}

\theoremstyle{plain}
\newtheorem{theorem}{Theorem}[section]

\newtheorem{lemma}[theorem]{Lemma}
\newtheorem{proposition}[theorem]{Proposition}

\newtheorem{conjecture}[theorem]{Conjecture}

\theoremstyle{definition}

\theoremstyle{remark}
\newtheorem{remark}[theorem]{Remark}

\newcommand{\N}{\mathbb{N}}
\newcommand{\Z}{\mathbb{Z}}
\newcommand{\R}{\mathbb{R}}

\usepackage{tcolorbox}

\newcommand{\floor}[1]{{\left\lfloor #1 \right\rfloor}}

\DeclareMathOperator{\E}{\mathbf{E}}
\renewcommand{\P}{\mathbf{P}}

\newcommand{\e}{\mathbf{e}}

\renewcommand{\bar}[1]{\overline{#1}}
\renewcommand{\tilde}[1]{\widetilde{#1}}
\renewcommand{\epsilon}{\varepsilon}
\renewcommand{\phi}{\varphi}

\title{The genealogy of an exactly solvable Ornstein-Uhlenbeck type branching process with selection}

\author{Aser Cortines\thanks{\texttt{aser.cortinespeixoto@math.uzh.ch}} \\
Universit\"at Z\"urich
\and
Bastien Mallein\thanks{\texttt{mallein@math.univ-paris13.fr}} \\
LAGA, Université Paris 13
}
\date{\today}

\renewcommand{\e}{\mathrm{e}}

\newcommand{\eq}{\mathrm{eq}}
\newcommand{\egaldistr}{\overset{(d)}{=}}

\begin{document}

\maketitle

\begin{abstract}
We study the genealogy of a solvable population model with $N$ particles on the real line which evolves according to a discrete-time branching process with selection. At each time step, every particle gives birth to children around $a$ times its current position, where $a>0$ is a parameter of the model. Then, the $N$ rightmost new-born children are selected to form the next generation. We show that the genealogical trees of the process converge to those of a Beta coalescent as $N \to \infty$. The process we consider can be seen as a toy-model version of a continuous-time branching process with selection, in which particles move according to independent Ornstein-Uhlenbeck processes. The parameter $a$ is akin to the pulling strength of the Ornstein-Uhlenbeck motion.
\end{abstract}

\section{Introduction}
\label{sec:introduction}

A \textit{branching-selection particle system} is a Markov process of particles on the real line that evolves through the repeated application of the two following steps:
\begin{description}
  \item[Branching step:] each particle currently alive in the process independently gives birth to children according to a point process whose law depends on the position of the particle.
  \item[Selection step:] some of the newborn children are selected to form the next generation and reproduce at the next \textit{branching step}, while the other particles are ``discarded'' from the process.
\end{description}
From a biological perspective, such models can be thought of as toy-models for the competition between individuals in a population evolving in an environment with limited resources. In this sense, the positions of particles (also seen as individuals) may be interpreted as their \textit{fitness}: individuals with large fitness have more propensity to reproduce and transfer their genetic advantage to their offspring.

Branching-selection particle systems are also of physical relevance. They are related to noisy reaction-diffusion phenomena and to the F-KPP equation \cite{BD97, BDMM07}, hence they are often used to describe the evolution of disordered systems having two homogeneous steady states. The prototypical example of such systems is the so-called $N$-branching random walk: in this process, particles make independently children around their current position, and at each step only the $N$ rightmost children are kept alive. Based on numerical simulations and the analysis of solvable models (see \cite{BD97, BDMM07}), it has been conjectured that many branching random walks with similar selection procedure satisfy universal properties. The cloud of particles travels at a deterministic speed $v_N$ that should satisfy
\begin{equation}\label{eq_Brunet_derrida_conjecture_speed}
  v_N - v_\infty = \frac{-\chi}{(\log N + 3 \log \log N + o(\log \log N))^2}.
\end{equation}
Moreover, the genealogical trees of such models should converge to those of a Bolthausen-Sznitman coalescent.

Some of these conjectures have been verified for individual models. Bérard and Gouéré~\cite{BeG} proved that the corrections to the speed $v-v_N$ of many $N$-branching random walks are of the order $(\log N)^{-2}$, in accordance with the conjectures from \cite{BDMM07}. These results haven been extended to branching random walks with different integrability conditions in \cite{BeM,CoG,Mal15b}, and to other related models in \cite{CoC15,Mal15a,Pai16}. Maillard \cite{Mai} proved that the barycentre of a $N$-branching Brownian motion converges, after proper scaling, to a  L\'evy process. This result indicates the validity of the second order corrections in \eqref{eq_Brunet_derrida_conjecture_speed}. Other results on the hydrodynamic limit of the shape of the front were obtained in \cite{DuR11,DMFP}.

As for the genealogical structure of such models, the authors in \cite{BBS13} show that the genealogy of a branching Brownian motion with quasi-critical absorption converges toward the Bolthausen-Sznitman coalescent. To the best of our knowledge this is the only example lying on the $N$-branching random walk universality class for which the conjectures from \cite{BDMM07} about the genealogy have been verified. Nevertheless results on related models \cite{BDMM07, Cor16, CoMa2017} indicate the robustness of this conjecture. We studied in \cite{CoMa2017} the so-called \emph{exponential model} from \cite{BDMM07}, which shares many common features with $N$-branching random walks, even though it does not belong to the same universality class. We showed that when the selection procedures favours the rightmost individuals, the genealogy of the process converges toward the Bolthausen--Sznitman coalescent. 
This result lead to the natural question: under which conditions do different genealogical behaviours arise? 

It is well known that the genealogy of neutral population models, such as Wright--Fisher and Moran models, converge to the Kingman coalescent \cite{MoS01}. On the other hand, if the rightmost particles are favoured by the selection procedure, then the Bolthausen--Sznitman coalescent is expected. Hence we look for coalescent processes arising in branching-selection particle systems, which interpolate between these two behaviours. Such a family appears in the context of Galton-Watson trees \cite{Sch03}, where the so-called Beta coalescent forms a $1$-parameter family interpolating between the Kingman and the Bolthausen--Sznitman coalescent.

In this paper, we consider a variant of the exponential model from \cite{BDMM07}, in which particles are subjected to a pulling force attracting them to zero. Let $N \in \N$ denote the size of the population and $a \in \R_+$ be a positive parameter governing the intensity of the attractive force. The process is defined as follows: it starts with $N$ particles scattered on the on the real line. At each discrete time $n$, every particle gives birth to children whose position are determined by independent Poisson point processes. More precisely, the offspring of an individual located at $x \in \R$ are positioned according to a Poisson point process with intensity $\e^{a x - y} \mathrm{d}y$ on $\R$. We then select the $N$ rightmost newborn individuals to form the next generation of the process. We call this process the $(N,a)$\emph{-exponential model}.

In the coming section, we will show that the $(N,a)$-exponential model is well-defined for all $n \in \N$, and that it is a reversible Markov process. Hence we can construct a version of it for all $n \in \Z$. Thanks to this bi-infinite construction, we define the \textit{ancestral partition process} $\Pi^N$ of the process as follows: for every $n \in \N$, let $\Pi^N_n$ be the partition of $\{1,\ldots,N\}$ such that $i$ and $j$ belong to the same block if and only if the $i$th and the $j$th rightmost particles at time $0$ share a common ancestor at time $-n$. Our main result concerns the asymptotic behaviour of $\Pi^N$ as $N \to \infty$.

\begin{theorem} \label{thm:main} As $N \to \infty$, we have:
\begin{enumerate}[a)]
  \item If $0<a<1/2$, then $(\Pi^N_\floor{t N}, t \geq 0)$ converges in law to the Kingman coalescent.
  \item If $a = 1/2$, then $(\Pi^N_\floor{tN/\log N}, t \geq 0)$ converges in law to the Kingman coalescent.
  \item If $1/2<a<1$, then $(\Pi^N_\floor{tN^{(1-a)/a}}, t \geq 0)$ converges in law to the Beta$(2-a^{-1},a^{-1})$-coalescent. 
  \item If $a=1$, then $(\Pi^N_\floor{t \log N}, t \geq 0)$ converges in law to the Bolthausen-Sznitman coalescent.
  \item If $a>1$, then $\Pi^N$ converges in law toward a discrete coalescent on $\N$.
\end{enumerate}
\end{theorem}

We refer to \cite[Example 3, p72]{Ber} for the precise definitions of the limiting processes, and Section~2.2 of the same book for the topology in which the above convergence takes place.

In view of the above result, $a=1$ marks a phases transition in the coalescent behaviour. We can also observe this transition in the dynamical behaviour of the cloud of particles. As long as $a < 1$, the cloud of particles remains within finite distance from 0. When $a=1$, we proved in \cite{CoMa2017} that it drifts toward~$\infty$ at positive speed. Finally, one can easily show that when  $a > 1$ the cloud moves away from $0$ at exponential rate. For $a < 1$, we have the following more precise estimates on the extremal positions of the cloud of particles.

\begin{proposition}
\label{prop_assymptotic_speed}
Given $a \in (0,1)$, we denote by $M_n$ and $m_n$ the largest and smallest positions at time $n$ in the $(N,a)$-exponential model. Then, if the initial conditions are bounded in $\mathrm{L}^1$ uniformly in $N$, we have
\[
  \lim_{N \to \infty }  \lim_{n \to \infty} (\E(M_n)-\log N) = \gamma - \frac{\log (1-a)}{1-a} \quad \text{and} \quad \lim_{N \to \infty} \lim_{n \to \infty} \E(m_n) = - \frac{\log (1-a)}{1-a},
\]
where $\gamma$ is the Euler-Mascheroni constant.
\end{proposition}

Proposition~\ref{prop_assymptotic_speed} shows that the cloud of particle in the $(N,a)$-exponential model is roughly of size $\log N$, which is typical in many branching selection particle systems.

The proofs of both Proposition~\ref{prop_assymptotic_speed} and Theorem~\ref{thm:main} rely on the observation that the distribution of the children at time $n+1$ is a Poisson point process with exponential intensity around the position of a unique fictional particle. This construction was introduced in \cite{BDMM07}, and further developed in \cite{CoMa2017}, from where we borrow the approach.

\paragraph*{Outline of the paper:} In the next section, we prove the two main results of the paper, namely, Theorem~\ref{thm:main} and Proposition~\ref{prop_assymptotic_speed}. In Section~\ref{sec:conjectures}, we define the \emph{branching Ornstein-Uhlenbeck processes} and discuss its relationship with the $(N,a)$-exponential model. 

\section{Proofs of the main results}

We start with the mathematical definition of the $(N,a)$-exponential model. Let $(\mathcal{P}_{n,j}, n \in \N, j \in \N)$ be an infinite array of i.i.d. copies of a Poisson point processes with intensity measure $\mathrm{e}^{-x} \mathrm{d}x$ and $X_1(0) > X_2(0) > \cdots > X_N(0)$ denote the ranked position of the particles at time $0$, that we shall assume distinct for the sake of simplicity. At each time $n$, the position of the $i$th largest individual is denoted by $X_n(i)$, and its children at time $n+1$ are distributed according to the point process $\mathcal{P}_{n+1,i}$ shifted by $a X_n(i)$. Then, we select the $N$-rightmost, new-born individuals to form the $(n+1)$th generation and denote by $X_{1}(n+1) > \cdots > X_N(n+1)$ their ranked positions. To keep track of the genealogical structure of the process, we shall label individuals with $A_{n+1}(k) \in \{1, \ldots N \}$ according the the index of their parents.

In other words, for each $n \in \N$ and $k \leq N$, we have that
\begin{enumerate}[i)]
\item \label{eq_position_n+1} $X_{n+1}(k)$ is the $k$th largest atom of the point process $\sum_{j=1}^N \sum_{p \in \mathcal{P}_{n+1,j}} \delta_{a X_n(j) + p}$,
\item $A_{n+1}(k)$ is the (unique) integer $j \leq N$ such that $X_{n+1}(k)-a X_n(j) \in \mathcal{P}_{n+1,j}$.
\end{enumerate}
It can be readily checked using standard properties of Poisson point processes that the above process is well defined for all $n$. Moreover, the law of $(X_{n+1}(k), k \leq N )$ from \eqref{eq_position_n+1} may be obtained as the $N$ largest points in a Poisson point process centered at  
\begin{equation}
  \label{eqn:defEquivalentPosition}
  X_n(\eq) := \log \left( \sum_{j=1}^N \mathrm{e}^{a X_n(j)} \right).
\end{equation}
Therefore, one may think of $X_n(\eq)$ as the ``equivalent'' position of the front, that is, a fictional particle that generates the entire front in generation $n+1$. In the next lemma we prove the above claim and characterize the (conditional) law of $A_n(\cdot)$.

\begin{lemma}
\label{lem:stationary}
The point processes $\sum_{j=1}^N \delta_{X_{n+1}(j) - X_n(\eq)}$, $n \in \N$ are i.i.d. with common distribution given by the $N$ rightmost points in a Poisson point process with intensity measure $\mathrm{e}^{-x} \mathrm{d}x$. Moreover, let $\mathcal{H} := \sigma \left( X_n(j), j \leq N, n \in \N \right)$ and $k_1, \ldots,k_N \in \Z_+$ such that $k_1 + \ldots + k_N =N$, then
\[
\P\left( A_{n+1}(j) = k_j;\, 1\leq j \leq N \middle| \mathcal{H} \right) = \prod_{j=1}^N \theta_n(k_j), 
\enskip \text{where} \qquad 
\theta_n(k) := \frac{\e^{aX_n(k)}}{\sum_{i=1}^N \e^{aX_n(i)}}.\]
\end{lemma}

\begin{proof}
This result is obtained with a reasoning similar to \cite[Proposition 1.3 and Lemma 1.6]{CoMa2017}, we will therefore only outline the main parts of the proof and omit the technical details.  

First, using the superposition property of Poisson point processes, we obtain that
\[
\sum_{j=1}^N \sum\limits_{p \in \mathcal{P}_{n+1,j}} \delta_{a X_n(j) + p}
\egaldistr 
\sum_{p \in \mathcal{P}} \delta_{X_n(\eq) + p},
\]
where $\mathcal{P}$ is a Poisson point process with intensity measure $\e^{-x}\mathrm{d}x$.

Next, we note that for all $i,j,n \in \N$ 
\begin{align*}
\P\left( A_{n+1}(j) = i  \middle| \, \mathcal{H} \, \right) 
& =  \P\left( X_{n+1}(j)-a X_n(i) \in \mathcal{P}_{n+1,i} \middle| \, \mathcal{H}  \,\right) \\
& = \frac{\e^{-X_{n+1}(j)-a X_n(j)}}{\sum_{k = 1}^N e^{-(X_{n+1}(j) - a X_n(k))}} = \frac{\e^{a X_n(j)}}{\sum_{k=1}^N \e^{a X_n(k)}}.
\end{align*}
Moreover, the $ A_{n+1}(j)$'s are (conditionally) independent, which concludes the proof.
\end{proof}

Using Lemma \ref{lem:stationary}, we observe that for any $n \in \N$, we have the following recursion equation:
\[
 X_{n+1}(\eq) - aX_n(\eq) = \log \left( \sum_{j=1}^N \e^{a (X_{n+1}(j) - X_n(\eq))} \right) \egaldistr  \log \left( \sum_{j=1}^N \e^{a \Delta_j} \right),
\]
where $(\Delta_j, j \in \N)$ are the ranked atoms of a Poisson point process with intensity measure $\e^{-x}\mathrm{d}x$. As a result, writing $(\xi_j, \, j \in \N)$ for an i.i.d. sequence of random variables with the same law as $X_1(\eq)-a X_0(\eq)$, we have
\begin{equation}
  \label{eqn:perpetuity}
  X_n(\eq) \egaldistr  \sum_{j=0}^{n-1} a^j \xi_{n-j} + a^n X_0(\eq).
\end{equation}
In other words, $(X_n(\eq), n \in \N)$ behaves as a perpetuity process.

\begin{proof}[Proof of Proposition~\ref{prop_assymptotic_speed}]
We recall that $M_n$ and $m_n$ denote respectively the positions of the rightmost and the leftmost particles at time $n$. Using Lemma~\ref{lem:stationary}, we observe that
\[
  M_n - X_{n-1}(\eq) \egaldistr p_1 \quad \text{and} \quad m_{n} - X_{n-1}(\eq) \egaldistr p_N,
\]
where $p_j$ is the $j$th largest atom in a Poisson point process with intensity $\e^{-x} \mathrm{d}x$. Hence, we have by standard Poisson computations that
\begin{equation}
  \label{eqn:liminfSup}
  \E(M_n) =\E(X_{n-1}(\eq)) + \gamma \quad \text{and} \quad \E(m_n) = \E(X_{n-1}(\eq)) - \psi(N)
\end{equation}
where $\psi(N):= \Gamma'(N)/\Gamma(N)$ is the digamma function, which satisfies $\psi(N) = \log N + \mathcal{O} (N^{-1})$ as $N\to \infty$. It is therefore enough to compute the asymptotic behaviour of $\E(X_{n}(\eq))$ as $n \to \infty$ and then $N \to \infty$ to conclude the proof.

For what follows in the proof, we shall assume that $a<1$ and $\E(X_0(\eq))<\infty$. Thus, taking the expected value in~\eqref{eqn:perpetuity}, one gets
\begin{equation}
  \label{eqn:limMed}
  \E(X_n(\eq)) = a^n \E(X_0(\eq)) + \frac{1- a^{n}}{1-a} \E(\xi_1),
 \quad\text{ yielding } 
 \lim_{n \to \infty} \E(X_n(\eq)) = \frac{1}{1-a} \E(\xi_1).
\end{equation}
It remains therefore to compute $\E(\xi_1)$. For this propose, we will first compute the Laplace transform $\Lambda(\lambda)$ of $\xi_1$, then recover its mean via $\E(\xi_1) = -(\log \Lambda)' (0)$.

By \cite[Proposition 4.2]{CoMa2017}, we remark that
\[
\e^{\xi_1} = \sum_{j=1}^N \e^{a \Delta_j} \egaldistr \e^{a Z_n} \sum_{j=1}^n \e^{a E_j},
\]
where $(E_j, j \geq 1)$ are i.i.d. exponential random variables with parameter $1$ and $Z_N$ is an independent random variable whose distribution has density $(N!)^{-1} \e^{-(N+1)x - \e^{-x}}$ with respect to the Lebesgue measure. Therefore, we have
\[
  \Lambda(\lambda) := \E\left( \e^{- \lambda \xi_1} \right) = \E\left( \left( \sum_{j=1}^N \e^{a E_j} \right)^{-\lambda} \right) \E\left( \e^{-\lambda a Z_N } \right).
\]
By direct computations, we obtain $\E\left( \e^{- \lambda a Z_N}\right) = \Gamma(N+1+a\lambda)/\Gamma(N+1)$ and
\begin{equation}
  \label{eq_I_1234321}
  \E\left( \left( \sum_{j=1}^N \e^{aE_j}\right)^{-\lambda} \right) = \frac{1}{\Gamma(\lambda)} \int_0^\infty t^{\lambda - 1} \E\left( e^{-t \sum_{j=1}^N \e^{aE_j}} \right) \mathrm{d}t = \frac{1}{\Gamma(\lambda)} \int_0^\infty t^{\lambda - 1} I(t)^N \mathrm{d}t,
\end{equation}
where $I(t)$ is the function defined by
\[
  I(t) := \E\left( \e^{-t \e^{a E_1}} \right) = \int_0^\infty \e^{-x} \e^{-t \e^{ax}} \mathrm{d}x = \frac{t^{1/a}}{a} \int_t^\infty u^{-(1+1/a)} \e^{-u } \mathrm{d}u.
\]

Making the change of variable $t = x/N$ in the right-hand side of \eqref{eq_I_1234321} one obtains
\[
 \E\left( \left( \sum_{j=1}^N \e^{aE_j}\right)^{-\lambda} \right) = 
 \frac{1}{N^\lambda} \int_0^\infty I(x/N)^N \frac{x^{\lambda-1} \mathrm{d}x }{\Gamma(\lambda)}   =: \frac{J_N(\lambda)}{N^\lambda},
\]
where $J_N(\lambda)$ is a smooth function such that $J_N(0) =1$. Collecting all pieces, we obtain that $\Lambda(\lambda) = J_N(\lambda) \frac{\Gamma(N+a\lambda+1)}{N^\lambda \Gamma(N +1)}$, which yields
\begin{equation}
  \label{eqn:expressionXi}
  \E\left( \xi_1 \right) = -\left( \log \Lambda \right)'(0) = \log N - a \psi(N+1) - J_N'(0),
\end{equation}
where we recall that $\psi(x) = \Gamma'(x) /\Gamma(x)$ is the digamma function and $(\log J_N)'(0)= \frac{J_N'(0)}{J_N(0)} = J_N'(0)$.

To compute $J_N'(0)$, we will take the $\lambda \to 0$ limit of 
\begin{equation}
  \label{integral:dominated_convergence}
  \frac{J_N(\lambda) - 1}{\lambda} = \frac{1}{\lambda \Gamma(\lambda)} \int_0^\infty x^{\lambda - 1} (I(x/N)^N - \e^{-x}) \mathrm{d}x = \frac{1}{\Gamma(\lambda + 1)} \int_0^\infty  x^{\lambda - 1} (I(x/N)^N - \e^{-x}) \mathrm{d}x.
\end{equation}
By definition, $I(t) \leq \e^{-t}$ for all $t \in \R_+$ which implies 
\begin{equation}
\label{eq_estimates_x>0}
 \big| x^{\lambda - 1} (I(x/N)^N - \e^{-x}) \big| \leq 2 \e^{-x}
 \qquad \text{for all $x\geq 1$ and $\lambda \in (0,1)$.}
\end{equation}
On the other hand, we observe that for all $t \in (0,1]$, we have
\begin{align*}
  I(t) &= \frac{t^{1/a}}{a} \left(\int_t^\infty u^{-(1+1/a)}(1 - u) \mathrm{d}u + \int_t^\infty u^{-(1+1/a)} (\e^{-u} - 1 + u) \mathrm{d}u \right) \\
  &= \frac{t^{1/a}}{a} \left( at^{-1/a} + \frac{t^{1 - 1/a}}{1 - 1/a} + \int_t^\infty u^{-(1+1/a)} (\e^{-u} - 1 + u) \mathrm{d}u \right)\\
  &= 1 - \frac{t}{1-a} + \mathcal{O}(t^{b}) \quad \text{as } t \to 0,
\end{align*}
for some $b>1$. Indeed, we have $t^{1/a} \int_1^\infty u^{-(1+1/a)} (\e^{-u} - 1 + u) \mathrm{d}u = \mathcal{O}(t^{1/a})$ as $t \to 0$, and relying on the fact that $(\e^{-u} - 1 + u) = \mathcal{O}(u^2)$ for all $u \leq 1$ one gets
\[
  t^{1/a} \int_t^1 u^{-(1+1/a)} (\e^{-u} - 1 + u) \mathrm{d}u = 
  \begin{cases}
    \mathcal{O}(t^{1/a}) & \text{ if } a > 1/2; \\
    \mathcal{O}(t^{1/a}) \log t & \text{ if } a = 1/2;\\
    \mathcal{O}(t^2) & \text{ if } a < 1/2.
  \end{cases}
\]
As a result, we obtain $|I_N(x/N)^N - \e^{-x}| \leq   \frac{a}{a-1}  x + C \left(\tfrac{x^b}{N^{b-1}} + x^2 \right)$ for all $x \leq 1$, where $C > 0$ is a constant not depending on $N$. Thus, one can find a (possibly larger) constant $\tilde{C}>0$ such that
\begin{equation}
\label{eq_estimates_x<0}
\big| x^{\lambda - 1} (I(x/N)^N - \e^{-x}) \big| \leq \tilde{C},
\qquad \text{for all $x < 1$.}
\end{equation}
Thanks to \eqref{eq_estimates_x>0} and \eqref{eq_estimates_x<0}, we can apply dominated convergence in \eqref{integral:dominated_convergence}, to obtain
\[
  J_N'(0) = \lim_{\lambda \to \infty}   \frac{J_N(\lambda) - 1}{\lambda} = \int_0^\infty x^{-1} (I_N(x/N)^N-\e^{-x}) \mathrm{d}x.
\]

Now, we plug the above in \eqref{eqn:expressionXi} and use the fact that $\lim_{N \to \infty} \psi(N+1) - \log N = 0$ to get
\begin{equation}
  \label{eqn:varExpressionXi}
  \E(\xi_1) = (1-a)\log N + \int_0^\infty x^{-1} (\e^{-x}-I_N(x/N)^N) \mathrm{d}x + o(1) \quad \text{as } N \to \infty.
\end{equation}
Finally, we notice that $I_N(x/N)^N$ tends to $\e^{-\frac{x}{1-a}}$ as $N\to \infty$. Therefore, we can rely again on \eqref{eq_estimates_x>0} and \eqref{eq_estimates_x<0}, to apply dominated convergence thereby obtaining  
\begin{align*}
  \lim_{N \to \infty} \int_0^\infty x^{-1}(\e^{-x}-I_N(x/N)^N)  \mathrm{d}x
  &= \int_0^\infty x^{-1} (\e^{-x} - \e^{- \frac{x}{1-a}})  \mathrm{d}x\\
  &= \int_0^\infty \int_1^{\frac{1}{1-a}} \e^{-ux} \mathrm{d}u \mathrm{d}x
  = \int_1^{\frac{1}{1-a}} \frac{\mathrm{d}u}{u}= -\log (1 - a),
\end{align*}
which, in sight of \eqref{eqn:liminfSup}, \eqref{eqn:limMed} and \eqref{eqn:varExpressionXi} concludes the proof.
\end{proof}

We now turn to the proof of Theorem \ref{thm:main}. We first use Lemma~\ref{lem:stationary} to construct a bi-infinite version $(\bar{X}_n,\bar{A}_n, n \in \Z)$ of the $(N,a)$-exponential model.
\begin{remark}
Let $(\mathcal{P}_n, n \in \Z)$ be i.i.d. Poisson point processes with intensity $\e^{-x} \mathrm{d}x$ and write $(p_{n,j}, j \in \N)$ for the sequence of atoms in $\mathcal{P}_n$ ranked in the decreasing order. We set $\bar{X}_0(\eq) = 0$, and for $n \in \Z \setminus \{ 0 \}$
\[
\bar{X}_n(\eq) =
\begin{cases}
  \sum_{j=1}^n \zeta_j a^{n-j} & \text{if $n \geq 1$}; \\
  -\sum_{j=-1}^n \zeta_{j} a^{n-j} & \text{if $n \leq 1$}, \\
\end{cases} 
\]
where $\zeta_n : = \log  \sum_{j=1}^N \e^{a p_{n,j}}$. The position of the $N$ particles $\bar{X}_n(1) > \ldots > \bar{X}_n(N)$ in generation $n \in \Z$ is then defined as
\[
\bar{X}_n(k) = \bar{X}_{n-1}(\eq) + p_{n,k}; \qquad k=1,\ldots N.
\] 
Next, we assign labels to the particles as follows: the particle $\bar{X}_n(k)$ bears the label $\bar{A}_n (k) \in \{1, \ldots N\}$ such that
\[
\textstyle
\P( \bar{A}_n (k) =j   \mid \mathcal{P}_{n-1}) = \e^{a p_{n-1, j}} \big/ \sum_{i=1}^{N} \e^{a p_{n-1, i}}.
\] 
\end{remark}

In the bi-infinite version $(\bar{X},\bar{A})$ of the $(N,a)$-exponential model, the family $(\bar{A}_n , n \in \Z)$ is i.i.d. We can now use the $\bar{A}$'s to reconstruct the \emph{ancestral partition process} of the process as follows: for every $n \in \N$, we say that $i$ and $j$ belong to the same block of $\Pi^N_n$ if
\[
   \bar{A}_{-n}(\bar{A}_{-(n-1)}(\ldots \bar{A}_{-1}(i))) = \bar{A}_{-n}(\bar{A}_{-(n-1)}(\ldots \bar{A}_{-1}(j))).
\]
This allows us to express the law of $\Pi^N$ in terms of a population dynamics with independent generation. Precisely, let
\[
  \theta_{n}(j) = \frac{\e^{a X_{-n}(j)}}{\sum_{k=1}^N \e^{a X_{-n}(k)}} = \frac{ \e^{a p_{n-1, j}} }{ \sum_{i=1}^{N} \e^{a p_{n-1, i}}} .
\]
Then conditionally on $(\theta_n(j), -n \in \N, j \leq N)$, each individual at generation $n \leq -1$ chooses its parent at generation $n-1$ independently at random, selecting the parent $j$ with probability $\theta_{n-1}(j)$. This is often called a Cannings model defined by a multinomial distribution with $N$ independent trials and (random) probabilities outcomes $(\theta_1(j), j \leq N)$ (see \cite[Section 2.2.3]{Ber} for a definition of such processes).

Thanks to this last observation, we finally prove Theorem \ref{thm:main}.
\begin{proof}[Proof of Theorem \ref{thm:main}]
Thinking of $\Pi^N$ as the \emph{ancestral partition process} of a Canning model, Lemma~\ref{lem:stationary} together with \cite[Proposition 4.2]{CoMa2017} yields 
\[
\{\theta_1(i), i \leq N\}:=  \left\{\frac{\e^{a\bar{X}_{-1}(i)}}{\sum_{k=1}^N \e^{a\bar{X}_{-1}(k)}}, i \leq N\right\} 
\egaldistr \left\{\frac{e^{aE_{i}}}{\sum_{k=1}^N e^{aE_{k}}}, i \leq N \right\},
\]
where  $(E_j, j \in \N)$ i.i.d. exponential random variables with mean $1$. Moreover, for any $y \geq 1$, we have $\P(\e^{a E_j} \geq y) = y^{-1/a}$. Therefore, applying \cite[Theorem 1.2]{Cor16}, we conclude the proof of Theorem \ref{thm:main}.
\end{proof}

\section{The branching Ornstein-Uhlenbeck processes}\label{sec:conjectures}

In this section, we draw a parallel between branching Ornstein-Uhlenbeck processes and the $(N,a)$-exponential model we have introduced here. We first recall that an Ornstein-Uhlenbeck process is a continuous time diffusion that solves the stochastic differential equation:
\begin{equation}
  \label{eqn:sde}
  \mathrm{d}X_t = - \mu X_t \mathrm{d}t + \sigma \mathrm{d}W_t,
\end{equation}
where $\mu \geq 0$ is the \emph{pulling strength} of the process, $\sigma>0$ is the diffusion coefficient and $W$ is a standard Brownian motion. Therefore, the branching Ornstein-Uhlenbeck process with parameters $(\beta,\mu,\sigma)$ is a continuous-time branching process, whose underlying motion is governed by $X_t$. It starts with a unique particle at $0$. This particle (or individual) evolves according to an Ornstein-Uhlenbeck process with parameters $(\mu, \sigma)$ for an independent exponential time with mean $\beta$. At that time, the individual splits into two children, that start independent branching Ornstein-Uhlenbeck processes from the position of their parents. Up to a space-time linear transforms, we may assume without loss of generality that $\beta = \sigma= 1$.

To the best of our knowledge, there is a reduced number of rigorous results about this branching process. The authors in \cite{AdM14,AdM15} study the behaviour of particles in the bulk and show that a CLT type result holds when the branching rate is weak (or analogously the pulling force is small), whereas a completely different behaviour takes place in the \emph{large branching rate} case. In \cite{Shi17}, a branching Ornstein-Uhlenbeck type process with infinite branching rate is introduced. Nevertheless, the behaviour of extreme particles and the genealogy associated to the process with selection remain open.
We believe there exists a straight connection between branching Ornstein-Uhlenbeck processes with selection and the $(N,a)$-exponential models. In both models particles are subjected to a pulling strength that depends linearly on the position of particles. In view of this connection and the results from the previous section, one can expect that the genealogy of branching Ornstein-Uhlenbeck processes with selection has non-trivial behaviour. 

As noticed in \cite{BDMM07}, the time step $n \mapsto n+1$ in the $(N,1)$-exponential model is associated to a $\mathcal{O}((\log N)^2)$ time step for the associated $N$-branching random walk. Hence, as the genealogical tree of the $(N,1)$-exponential model converges in the $(\log N)$ scale, the genealogical tree of the $N$-branching random walk should converge in the $(\log N)^3$ scale. These conjecture were verified in \cite{BBS13}  to the case of branching Brownian motion with quasi-critical absorption.  Similarly, a time step in the $(N,a)$-exponential model should correspond to a time interval of the order $(\log N)^2$ in the evolution of the branching Ornstein-Uhlenbeck process with selection. If this connection is precise, it is reasonable to choose a pulling force $\mu_N$ depending on the size of the system. A reasonable choice would then be $\gamma/(\log N)^2$ with $\gamma > 0$ a parameter. We notice that an Ornstein-Uhlenbeck process $X$ with pulling strength $\mu_N$ and starting position $x$ satisfies $\E(X_{(\log N)^2})=x \e^{-\gamma}$.

We ran a few simulations that reinforce the above heuristics. We considered discrete-time/space branching-selection particles systems which mimic branching Ornstein-Uhlenbeck processes. The results are displayed in Figure~\ref{figure_simu}. We observe that when the pulling strength has the order $(\log N)^{-2}$, then the average coalescent time of two individuals chosen uniformly at random seems to grow polynomial with $N$. This type of behaviour is compatible with Beta coalescent. This leads us to the following conjecture.

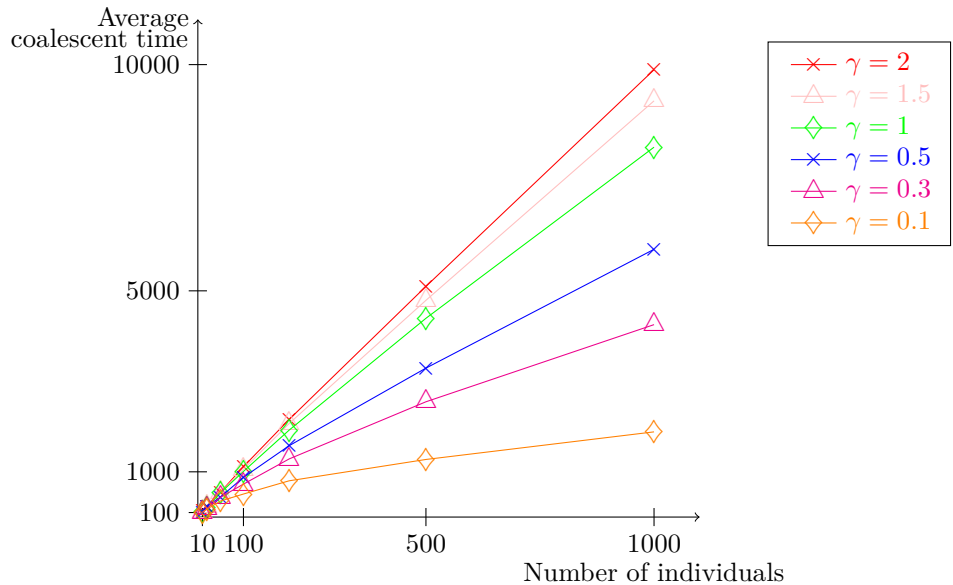
\begin{figure} \label{figure_simu}
\centering
\begin{tikzpicture}[xscale=0.06,yscale=0.06]
\draw [->] (0,0) -- (0,110) node[left] {Average};
\draw (0,106) node[left] {coalescent time};
\draw [->] (0,0) -- (110,0);
\draw (100,-8) node[below] {Number of individuals};
\draw (1,-2) node[below] {10} -- (1,2);
\draw (10,-2) node[below] {100} -- (10,2);
\draw (50,-2) node[below] {500} -- (50,2);
\draw (100,-2) node[below] {1000} -- (100,2);

\draw (-2,1) node[left] {100} -- (2,1);
\draw (-2,10) node[left] {1000} -- (2,10);
\draw (-2,50) node[left] {5000} -- (2,50);
\draw (-2,100) node[left] {10000} -- (2,100);

\draw [color=red] (130,100) -- (135,100) node {$\times$} -- (140,100) node[right] {$\gamma=2$};
\draw [color=pink] (130,93) -- (135,93) node {$\triangle$}-- (140,93) node[right] {$\gamma=1.5$};
\draw [color=green] (130,86) -- (135,86) node {$\diamondsuit$}-- (140,86) node[right] {$\gamma=1$};
\draw [color=blue] (130,79) -- (135,79) node {$\times$}-- (140,79) node[right] {$\gamma=0.5$};
\draw [color=magenta] (130,72) -- (135,72) node {$\triangle$}-- (140,72) node[right] {$\gamma=0.3$};
\draw [color=orange] (130,65) -- (135,65) node {$\diamondsuit$}-- (140,65) node[right] {$\gamma=0.1$};

\draw (125,105) -- (165,105) -- (165,60) -- (125,60) -- cycle;

\draw [color=red] (1,1.40688) node {$\times$} -- (2,2.35171) node {$\times$} -- (5,5.56506) node {$\times$} -- (10, 11.13705) node {$\times$} -- (20,21.5576) node {$\times$} -- (50,50.97692) node {$\times$} -- (100,98.86282) node {$\times$}; 
\draw [color = pink] (1,1.16) node {$\triangle$} -- (2,2.44) node {$\triangle$} -- (5,5.52) node {$\triangle$} -- (10,10.56) node {$\triangle$} -- (20,20.66) node {$\triangle$} -- (50,47.90) node {$\triangle$} -- (100,91.99) node {$\triangle$}; 
\draw [color=green] (1,1.25) node {$\diamondsuit$} -- (2,2.12) node {$\diamondsuit$} -- (5,5.4) node {$\diamondsuit$} -- (10,10.09) node {$\diamondsuit$} -- (20,19.20) node {$\diamondsuit$} -- (50,43.93) node {$\diamondsuit$} -- (100,81.75) node {$\diamondsuit$}; 
\draw [color = blue] (1,1.01) node {$\times$} -- (2,2.37) node {$\times$} -- (5,4.46) node {$\times$} -- (10,8.78) node {$\times$} -- (20,15.76) node {$\times$} -- (50,32.92) node {$\times$} -- (100,59.18) node {$\times$}; 
\draw [color = magenta] (1,0.97) node {$\triangle$} -- (2,1.89) node {$\triangle$} -- (5,4.41) node {$\triangle$} -- (10,7.28) node {$\triangle$} -- (20,12.84) node {$\triangle$} -- (50,25.48) node {$\triangle$} -- (100,42.53) node {$\triangle$}; 
\draw [color=orange] (1,1.1) node {$\diamondsuit$} -- (2,1.75) node {$\diamondsuit$} -- (5,3.56) node {$\diamondsuit$} -- (10,5.09) node {$\diamondsuit$} -- (20,8.01) node {$\diamondsuit$} -- (50,12.73) node {$\diamondsuit$} -- (100,18.83) node {$\diamondsuit$}; 

\end{tikzpicture}
\caption{Average age of the most recent common ancestor of two individuals selected at random}
\end{figure}

\begin{conjecture}
Let $\gamma > 0$ and $N \in \N$, there exists a sequence $(\rho_N)$ which is $c_\gamma$-regularly varying, with $c_\gamma \in [0,1]$ such that $(\Pi^N_\floor{\rho_N t},t \geq 0)$ converges toward a Beta($1-c_\gamma,1+c_\gamma)$-coalescent.
\end{conjecture}

Roughly speaking, this conjectures state that a branching Ornstein-Uhlenbeck process with pulling strength $\gamma (\log N)^{-2}$ and selection of the $N$ rightmost individuals can be associated to a $(N,(c_\gamma+1)^{-1})$-exponential model. Even though we were not able to push the simulations far enough to guess the function $c_\gamma$, it is worth noting that it seems to be close to 1 for $\gamma > 1$, and that $c_{0.1} \approx 0.5$.

\paragraph*{Acknowledgments:} 
We would like to thank Julien Berestycki and \'Eric Brunet for their comments, advices and interest in this work.

\noindent
The work of A.C. is supported by the Swiss National Science Foundation 200021\underline{{ }{ }}163170.

\bibliographystyle{plain}

\bibliography{biblio}

\end{document}